\documentclass[11pt]{article}

\usepackage{verbatim}
\usepackage{amsfonts}
\usepackage{amsthm}
\usepackage{amssymb}
\usepackage{amsmath}
\usepackage[latin1]{inputenc}
\usepackage{graphicx}
\usepackage{epstopdf}
\theoremstyle{plain}

\newtheorem{thm}{Theorem}[section]
\newtheorem{prop}[thm]{Proposition}
\newtheorem{lem}[thm]{Lemma}

\newtheorem{cor}[thm]{Corollary}
\theoremstyle{definition}

\newtheorem{rem}[thm]{Remark}
\newtheorem{example}[thm]{Example}
\newtheorem{problem}[thm]{Problem}

\newcommand{\g}{{\gamma}}

\newcommand{\be}{{\beta}}

\newcommand{\s}{{\sigma}}

\newcommand{\NDes}{\operatorname{NDes}}

\newcommand{\maj}{\operatorname{maj}}
\newcommand{\fmaj}{\operatorname{fmaj}}

\newcommand{\inv}{\operatorname{inv}}
\newcommand{\me}{\operatorname{N}}
\newcommand{\DDes}{\operatorname{DDes}}
\newcommand{\nmaj}{\operatorname{nmaj}}
\newcommand{\dmaj}{\operatorname{dmaj}}
\newcommand{\des}{\operatorname{des}}
\newcommand{\ndes}{\operatorname{ndes}}
\newcommand{\ddes}{\operatorname{ddes}}
\newcommand{\Des}{\operatorname{Des}}

\newcommand{\NN}{\mathbb{N}}
\newcommand{\ZZ}{\mathbb{Z}}
\newcommand{\PP}{\mathbb{P}}

\begin{document}

\title{Equidistribution of negative statistics and quotients of Coxeter groups of type $B$ and $D$}
\author{Riccardo Biagioli \footnote{Dedicated to the memory of my friend and colleague Giulio Minervini.}\\
Universit\'e de Lyon, Universit\'e Lyon 1\\
Institut Camille Jordan - UMR 5208 du CNRS\\
43, boulevard du 11 novembre 1918\\
F - 69622 Villeurbanne Cedex
}

\date{}
\maketitle

\begin{abstract}
We generalize some identities and $q$-identities previously known for the symmetric group to Coxeter groups of type $B$ and $D$. The extended results include theorems of Foata and Sch\"utzenberger, Gessel, and Roselle on various distributions of inversion number, major index, and descent number. In order to show our results we provide caracterizations of the systems of minimal coset representatives of Coxeter groups of type $B$ and $D$.
\end{abstract}

\section{Introduction}

A well known theorem of MacMahon \cite{MM} shows that the {\em length function} and the {\em major index} are equidistributed over the symmetric group $S_n$. 
We recall that the length of a permutation $\sigma \in S_n$ is given by the {\em number of inversions}, denoted $\inv(\s):=|\{(i,j) \mid i<j  , \; \s(i) > \s(j) \}|$, and the major index of $\s$  is the sum of all its {\em descents}. More precisely,
$$ \maj(\sigma):=\sum_{i \in \Des(\sigma)} i, $$ 
where $\Des(\s):=\{i \in [n-1] \mid  \s(i)>\s(i+1)\}$. Foata gave a bijective proof of this equidistribution theorem in \cite{F}.  He studied further his bijection and together with 
%He defined a bijection $\varphi: S_n \rightarrow S_n$, with the following two properties 
%\begin{equation*}\label{majinv}
%\maj(\s)=\inv(\varphi(\s)), \;\; {\rm and} \;\; \Des(\s^{-1})=\Des(\varphi(\s)^{-1}).
%\end{equation*} 
Sch\"utzenberger derived the two following results \cite{FS}. The first one is a refinement of MacMahon's theorem, asserting the equidistribution of major index and number of inversions over descent classes. 

\begin{thm}[Foata-Sch\"utzenberger]\label{FS1}
Let $M=\{m_1,\ldots, m_t\}_< \subseteq \{1,\ldots, n-1\}$. Then
\[ \sum_{\{\s \in S_n \mid \Des(\s^{-1})=M\}} q^{\maj(\s)}= \sum_{\{\s \in S_n \mid \Des(\s^{-1})=M\}} q^{\inv(\s)}\]
\end{thm}
The second one concerns the symmetry of the distribution of the major index and the inversion number over the symmetric group.
\begin{thm}[Foata-Sch\"utzenberger] \label{FS2} The pairs of statistics $(\maj,\inv)$ and $(\inv,\maj)$ have the same distribution on $S_n$, namely
\[S_n(t,q):=\sum_{\s \in S_n} t^{\maj(\s)}q^{\inv(\s)}=\sum_{\s \in S_n} t^{\inv(\s)}q^{\maj(\s)}.\]
\end{thm}

Theorem \ref{FS1} has been extensively studied and generalized in many ways in the last three decades. Nevertheless, it still receives a lot of attention as shown by two recent papers of Hivert, Novelli, and Thibon \cite{HNT}, and of Adin, Brenti, and Roichman \cite{ABR2}, where a multivariate generalization and an extension  to the hyperoctahedral group of it are provided. In the latter paper, the problem of finding an analogue of this Foata-Sch\"utzenberger theorem for the Coxeter groups of type $D$ is proposed \cite[Problem 5.6]{ABR2}.

In this paper we answer this question. Actually, we show that the {\em negative major indices} ``$\nmaj$", introduced in \cite{ABR1} on Coxeter groups of type $B$, and ``$\dmaj$", defined in \cite{B} on Coxeter groups of type $D$, give generalizations of the first and second Foata-Sch\"utzenberger identities to $B_n$ and $D_n$.  In our analysis we derive nice relations among {\em quotients}, or sets of minimal coset representatives, of $B_n$ and $D_n$ that are interesting in their own. Explicit maps between these quotients are shown, and used to compute some generating functions. 

Finally, we use our results, and the {\em negative descent numbers}, to give generalizations to $B_n$ and $D_n$ of two classical $q$-identities. The first one, due to Roselle \cite{Ros} (see also Rawlings \cite[(2.4)]{Raw}), is the generating function of the inversion number and major index over the symmetric group: for undefined notation see next section.
\begin{thm}[Roselle]\label{Roselle}
$$ \sum_{n\geq 0} S_n(t,q) \frac{u^n}{(t;t)_n (q;q)_n}=\frac{1}{(u;t,q)_{\infty,\infty}},$$
\end{thm}
\noindent where $S_0(t,q)=1$.  The second one is the trivariate distribution of inversion number, major index, and {\em number of descents},  due to Gessel \cite[Theorem 8.4]{Ges}, (see also \cite{GG}).
\begin{thm}[Gessel]\label{Gessel}
$$\sum_{n\geq 0} \frac{u^n}{[n]_q !} \frac{\sum_{\s \in S_n} t^{\maj(\s)} q^{\inv(\s)} p^{\des(\s)}}{(t;q)_{n+1}}=\sum_{k \geq 0} p^{k} e[u]_q e[tu]_q\cdots e[t^k u]_q.$$
\end{thm}

\section{Preliminaries and notation}

In this section we give some definitions, notation and results that
will be used in the rest of this work. For $n \in \NN$ we let $[n]:= \{ 1,2, \ldots , n \} $ (where $[0]:= \emptyset $). Given $n, m \in \ZZ, \; n \leq m$, we let $[n,m]:=\{n,n+1, \ldots, m \}.$ We let $\PP:=\{1,2,3,\ldots\}.$
The cardinality of a set $A$ will be denoted by $|A|$ and we let ${[n] \choose 2}:=\{S \subseteq [n] \mid |S|=2 \}.$ Given a set $A$, we denote $A_<:=\{a_1,a_2,\ldots\}$ where $a_1<a_2<\ldots$.

For our study we need notation for $q$-analogs of the factorial, binomial coefficient, and multinomial coefficient. These are defined by the following expressions
$$[n]_q:=1+q+q^2+\ldots + q^{n-1}; \;\;\;\;\; [n]_q!:=[n]_q[n-1]_q\cdots[2]_q[1]_q;$$
$$\begin{bmatrix} n \\ m \end{bmatrix}_q:=\frac{[n]_q!}{[m]_q![n-m]_q!}; \;\;\;\;\;\;
\begin{bmatrix} n \\ m_1, \; m_2,  \ldots, m_t \end{bmatrix}_q:=\frac{[n]_q!}{[m_1]_q![m_2]_q! \cdots [m_t]_q!}.$$
As usual we let
\begin{eqnarray*}
(a;q)_0&:=&1 \\
(a;q)_n&:=&(1-a)(1-aq)\cdots (1-aq^{n-1})\\
(a;q)_\infty&:=&\prod_{n\geq 1}(1-aq^{n-1}).
\end{eqnarray*}
Moreover, for $r, s \in \NN$ we let
$$(a;t,q)_{r,s}:=\left\{\begin{array}{ll} \;\; 1 & {\rm  if} \ r \ {\rm or} \ s \ {\rm are \ zero}\\
						{\displaystyle\prod_{1\leq i\leq r}\prod_{1\leq j\leq s} (1-at^{i-1}q^{j-1})} & {\rm if} \ r,s \geq 1\end{array}\right.,$$
and
$$(a;t,q)_{\infty,\infty}:=\prod_{i\geq 1}\prod_{j\geq 1}(1-at^{i-1}q^{j-1}).$$
Finally, 
$$e[u]_q:=\sum_{n\geq 0}\frac{u^n}{[n]_q !},$$
is the $q$-analogue of the exponential function. The following $q$-binomial theorem is well known (see e.g. \cite{An})
\begin{thm}\label{qbinomial}
$$(-xq;q)_n=\sum_{m=0}^n \begin{bmatrix} n \\ m \end{bmatrix}_q q^{{m+1 \choose 2}}x^m.$$
\end{thm}

\subsection{Coxeter groups of type $B$ and $D$}

We denote by $B_{n}$ the group of all bijections $\be$ of the set
$[-n,n]\setminus \{0\}$ onto itself such that
\[\be(-i)=-\be(i)\]
for all $i \in [-n,n]\setminus \{0\}$, with composition as the
group operation. This group is usually known as the group of {\em
signed permutations} on $[n]$, or as the {\em hyperoctahedral
group} of rank $n$.  If
$\be \in B_{n}$ then we write $\be=[\be(1),\dots,\be(n)]$ and we call this the
{\em window} notation of $\be$. As set of generators for $B_n$ we
take $S_B:=\{s_1^B,\ldots,s_{n-1}^B,s_0^B\}$ where for $i
\in[n-1]$
\[s_i^B:=[1,\ldots,i-1,i+1,i,i+2,\ldots,n] \;\; {\rm and} \;\; s_0^B:=[-1,2,\ldots,n].\]
It is well known that $(B_n,S_B)$ is a Coxeter system of type $B$ (see e.g., \cite[\S 8.1]{BB}).
\begin{figure}[htdb!]
\centering
\includegraphics[scale=.5]{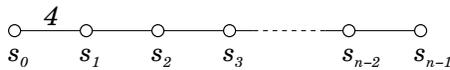}
\caption{The Dynkin diagram of $B_n$}\label{graphB}
\end{figure}

To give an explicit combinatorial description of the length function $\ell_B$ of $B_n$ with respect to $S_B$, we need the following statistics.
For $\be \in B_n$ we let
\begin{eqnarray*}
\me_1(\be) & := &   |\{i \in [n] \mid \be(i)<0\}|, \ \ {\rm and}\\
\me_{2}(\be)  &:=&  \left|\left\{\{i,j\} \in {[n] \choose 2} \mid \be(i)+\be(j)<0 \right\}\right|.
\end{eqnarray*}
Note that, if $\be \in B_n$, 
\begin{equation} \label{som}
\me_{1}(\be)+\me_{2}(\be)=-\sum_{\{i \in [n] \mid \be(i)<0\}} \be(i).
\end{equation}
\noindent For example if $\be=[-3,1,-6,2,-4,-5] \in B_{6}$ then
$\me_1(\be)=4$, and $\me_2(\be)=14$. 
\smallskip

The following characterizations of the length function, and of the right descent set of $\be \in B_n$ are well known \cite{BB}.

\begin{prop} Let $\be \in B_n$. Then
\begin{eqnarray*}
\ell_B(\be)&=&\inv(\be)+\me_1(\be)+\me_2(\be), \ \ and\\
\Des_B(\be)&=&\{i \in [0,n-1] \mid \be(i) > \be(i+1)\},
\end{eqnarray*}
where  $\be(0):=0$.
\end{prop}

\bigskip

We denote by $D_{n}$ the subgroup of $B_{n}$ consisting of all the
signed permutations having an even number of negative entries in
their window notation, more precisely
\[ D_{n} := \{\g \in B_{n} \mid \me_{1}(\g)\equiv 0 \; ({\rm mod} \; 2 ) \}. \]
It is usually called the {\em even-signed permutation group}. As a set of generators for $D_n$ we take
$S_D:=\{s_{0}^D,s_{1}^D,\dots,s_{n-1}^D\}$ where for $i \in [n-1]$
\[s_i^D:=s_i^B \;\; {\rm and} \;\; s_{0}^D:=[-2,-1,3,\ldots,n].\]
\begin{figure}[htdb!]
\centering
\includegraphics[scale=.5]{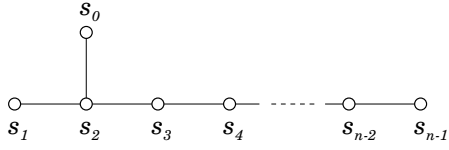}
\caption{The Dynkin diagram of $D_n$}
\label{graphD}
\end{figure}

There is a well known direct combinatorial way to compute the length, and the right descent set of $\g \in D_{n}$, (see, e.g., \cite[\S 8.2]{BB}). 
\begin{prop}Let $\g \in D_n$. Then
\begin{eqnarray*}\label{lD} 
\ell_D(\g)&=&\inv(\g) + \me_{2}(\g),\; \; {\rm and} \\
\Des_D(\g)&=&\{i \in [0,n-1] \mid \g(i)>\g(i+1)\},
\end{eqnarray*}
where $\g(0):=-\g(2)$.
\end{prop}

\subsection{Negative statistics}
 
In \cite{ABR1}, Adin, Brenti and Roichman introduced the following statistics on $B_n$. 
For $\be \in B_{n}$ let 
\[\NDes(\be):=\Des(\be) \biguplus \{-\be(i) \mid  \be(i)<0 \},\] 
and define
\begin{eqnarray*}
\nmaj(\be):= \sum_{i \in \NDes(\be)}i , \ \ {\rm and} \ \
\ndes(\be):= |\NDes(\be)|.
\end{eqnarray*}
It follows from (\ref{som}) that
\begin{eqnarray}
\label{nmaj} \nmaj(\be)&=&\maj(\be)+\me_1(\be) +\me_2(\be), \ \ {\rm and} \\
\label{ndes} \ndes(\be)&=&\des(\be)+\me_1(\be).
\end{eqnarray}
For the element $\be=[-3,1,-6,2,-4,-5] \in B_{6}$, $\nmaj(\be)=29$, and $\ndes(\be)=7$.

\bigskip

In \cite{B}, a notion of descent multiset for $\g\in D_n$ is introduced
\[\DDes(\g):=\Des(\g) \biguplus \{ -\g(i)-1 \mid \g(i)<0 \} \setminus \{0\},\] 
and the following statistics are defined 
\begin{eqnarray*}
\dmaj(\g):= \sum_{i \in \DDes(\g)}i ,  \ \ {\rm and} \ \ 
\ddes(\g):= |\DDes(\g)|.
\end{eqnarray*}
It easily follows that
\begin{eqnarray}
\label{dmaj} \dmaj(\g)&=& \maj(\g) + \me_{2}(\g), \ \ {\rm and}\\
\label{ddes} \ddes(\g)&=&\des(\g)+ \me_1(\g)+\epsilon(\g),
\end{eqnarray}
where
\begin{equation}
\epsilon(\g) := \left \{\begin{array}{ll}
-1 & \mbox{if $1 \not \in \g([n])$} \\
 0 & \mbox{if $1  \in \g([n])$}.
\end{array} \right. 
\end{equation}
For example if $\g=[-4,1,3,-5,-2,-6] \in D_{6}$ 
then $\dmaj(\g)=21$, and $\ddes(\g)=5$.
\bigskip

The statistics $\nmaj$ and $\dmaj$ are usually called {\em negative major indices}; $\ndes$ and $\ddes$ {\em negative descent numbers} for $B_n$ and $D_n$, respectively. The negative major indices are Mahonian statistics, namely they are equidistributed with the length over the group,
$$ \sum_{\be \in B_{n}} q^{\nmaj(\be)} = \sum_{\be \in B_{n}} q^{\ell_B(\g)}, \;\;\;\; {\rm and} \;\;\;\; \sum_{\g \in D_{n}} q^{\dmaj(\g)} = \sum_{\g \in D_{n}} q^{\ell_D(\g)}.$$
The pairs $(\ndes,\nmaj)$ and $(\ddes,\dmaj)$ give generalizations to $B_n$ and $D_n$ of a famous identity of Carlitz, see \cite[Theorem 3.2]{ABR1}, and \cite[Theorem  3.4]{B}.

\subsection{Quotients of Coxeter groups}

To show some of the next results we will need of the following decomposition that comes
from the general theory of Coxeter group. We refer the reader to \cite{BB} for any undefined notation. 

Let $(W,S)$ be a Coxeter
system, for $J \subseteq S$ we let $W_{J}$ be the {\em parabolic subgroup} of $W$
generated by $J$, and
$$
W^{J} := \{ w \in W \mid \ell(ws) > \ell(w) \; \; \mbox {for all} \; \; s \in J \},
$$
the set of minimal left coset representatives of $W_{J}$, or the {\em (right) quotient}. The quotient $W^{J}$ is a poset
according to the Bruhat order. The following is well known (see \cite[\S 2.4]{BB}).
\begin{prop}
\label{2.1} Let $(W,S)$ be a Coxeter system, and let $J \subseteq S$. Then:
\begin{itemize}
\item[i)] Every $w \in W$ has a unique factorization
$w=w^{J}w_{J}$ such that $w^{J} \in W^{J}$ and $w_{J} \in W_{J}$.
\item[ii)] For this factorization $\ell(w)=\ell(w^{J})+\ell(w_{J})$.
\end{itemize}
\end{prop}

As a first application of this decomposition to the groups $B_n$ (and $D_n$), 
let us consider the parabolic subgroup generated by $J:=S_B \setminus \{s_0^B\}$. In this case, by looking at the Dynkin diagram in Figure \ref{graphB}, we obtain that $B_J=S_n$. Moreover it is not hard to see that
\begin{eqnarray}
\label{T} B^J:=B_n^J=\{u \in B_{n} \mid u(1)<u(2)<\ldots<u(n) \}.
\end{eqnarray}
Hence from Proposition \ref{2.1} we get
\begin{equation}
\label{decomp} B_{n}= \biguplus_{\s \in S_{n}}\{ u \s \mid u \in B^J\},
\end{equation}
where $\biguplus$ denotes disjoint union. Note that in the case $D_n$, for $J:=S_D \setminus \{s_0^D\}$, a similar decomposition holds,
 \begin{equation*}
 D_{n}= \biguplus_{\s \in S_{n}}\{ u \s \mid u \in D^J\},
\end{equation*}
where once again $D_J=S_n$, and $D^J=\{u \in D_{n} \mid u(1)<u(2)<\ldots<u(n) \}.$

\begin{rem}
	The construction or right quotient can be mirrored, by considering {\em left descents}. Let $J\subseteq S$. A {\em left quotient} of $W$ is defined by
	$$ ^J W:=\{w \in W \mid \ell(s w)>\ell(w)\ {\rm for \ all} \ s \in J \}.$$
	Proposition \ref{2.1} holds for left quotients too, but the factorization in $i)$ becomes $w=w_J \cdot ^J\!w$, with $^J\! w \in$ $^J\! W$.
	Left and right quotients are isomorphic posets, by means of the inversion map. In the next section, we will work with subsets of $B_n$ and $D_n$ that are left quotients. They are called descent classes for reasons that will be immediately clear.
\end{rem}

\section{Combinatorial description of descent classes}

Let us fix a subset of descents $M:=\{m_1, m_2, \ldots, m_t\}_< \subseteq [0,n-1]$. The set
\begin{equation}\label{Bdescentclass}
B(M):=\{\be \in B_n \mid \Des_B(\be^{-1}) \subseteq M\},
\end{equation}
is usually called a $B$-{\em descent class}. Note that this set is nothing but a left quotient of $B_n$. More precisely, it is the one corresponding to the subset $J=S \setminus \tilde{M}$, where $\tilde{M}:=\{s_i \mid i \in M\}$. The following result can be found in \cite[Lemma 4.1]{ABR2}.

\begin{lem}\label{Bcombclass} Let $\be \in B_n$, and $M=\{m_1,\ldots,m_t\}_{<}\subseteq [0,n-1]$. Let $m_{t+1}:=n$. Then $\Des_B(\be^{-1})\subseteq M$ if and only if there exist (unique) integers $r_1,\ldots, r_t$ satisfying $m_i \leq r_i \leq m_{i+1}$ for all $i$, and such that $\be$ is a shuffle of the following increasing sequences:
\begin{equation}\label{shuffleB}
	\begin{array}{l}
	(1,2,\ldots,m_1), \\
	\left(-r_1,-r_1+1,\ldots,-(m_1+1)\right), \\
	(r_1+1,r_1+2,\ldots,m_2), \\
	\hspace{2cm}  \vdots \\
	\big(-r_t,-r_t+1,\ldots,-(m_t+1)\big), \\
	(r_t+1,r_t+2,\ldots,n).
	\end{array}
\end{equation}
Some of these sequences may be empty, if $r_i=m_i$ or $r_i=m_{i+1}$ for some $i$, or if $m_i=0$.
\end{lem}

The following one is an explicit description of $D$-descent classes 
\begin{equation}\label{Ddescentclass}
D(M):=\{\g \in D_n \mid \Des_D(\g^{-1}) \subseteq M\}.
\end{equation}
\begin{lem}\label{Dcombclass} Let $\g \in D_n$, and $M=\{m_1,\ldots,m_t\}_{<}\subseteq [0,n-1]$. Let $m_{t+1}:=n$.  Then $\Des_D(\g^{-1})\subseteq M$ if and only if there exist (unique) integers $r_1,\ldots, r_t$ satisfying $m_i \leq r_i \leq m_{i+1}$ for all $i$, and such that $\g$ is a shuffle of the following increasing sequences. There are three cases, and six possible ``blocks" of sequences.
\begin{itemize}
\item[1)] If $0 \in M$: $(m_1=0)$
\begin{equation}\label{shuffleD-O}
	\begin{array}{l}
	(-r_1,-r_1+1,\ldots,-2,-1), \\
	(r_1+1,r_1+2,\ldots,m_2),  \\
	\hspace{2cm} \vdots \\
	\big(-r_t,-r_t+1,\ldots,-(m_t+1)\big),\\ 
	(r_t+1,r_t+2,\ldots,n), \\
	\end{array}
\end{equation}
with \ {$\displaystyle \sum_{i=1}^t(r_i-m_i) \equiv 0 \; ({\rm mod} \ 2)$}.

\item[2)] If $0, 1\not\in M$: (note $m_1\geq 2$)
\begin{equation}\label{shuffleD-01}
	\begin{array}{lll}
	(1,2,\ldots,m_1), & (-1,2,\ldots,m_1) &\\
	\left(-r_1,-r_1+1,\ldots,-(m_1+1)\right), & (-r_1,-r_1+1,\ldots,-(m_1+1)) &\\
	(r_1+1,r_1+2,\ldots,m_2), & (r_1+1,r_1+2,\ldots,m_2) & \\
	\hspace{2cm} \vdots &  \hspace{2cm} \vdots \\
	\big(-r_t,-r_t+1,\ldots,-(m_t+1)\big),& \big(-r_t,-r_t+1,\ldots,-(m_t+1)\big) &\\ 
	(r_t+1,r_t+2,\ldots,n) & (r_t+1,r_t+2,\ldots,n)&
	\end{array}
\end{equation}
with
$$
\begin{array}{lll}
{\displaystyle \sum_{i=1}^t(r_i-m_i) \equiv 0 \; ({\rm mod} \ 2)}; &  {\displaystyle \sum_{i=1}^t(r_i-m_i) \equiv 1 \; ({\rm mod} \ 2).} & 
\end{array}$$
\item[3)] If $0 \not\in M$ and $1 \in M$: (note $m_2\geq 2$, and $r_1\geq 2$)
\begin{small}
\begin{equation}\label{shuffleD-1}
	\begin{array}{llll}
	 (1) &  & (-r_1,\ldots,-2,1) & \\
	(2,3,\ldots,m_2), & (-1,2,3,\ldots,m_2) & (r_1+1,r_1+2,\ldots,m_2)&\\
	\big(-r_2,-r_2+1,\ldots,-(m_2+1)\big), & \big(-r_2,-r_2+1,\ldots,-(m_2+1)\big), & \big(-r_2,-r_2+1,\ldots,-(m_2+1)\big) & \\
	(r_2+1,r_2+2,\ldots,m_3), & (r_2+1,r_2+2,\ldots,m_3), & (r_2+1,r_2+2,\ldots,m_3) & \\
	\hspace{2cm}  \vdots &  \hspace{2cm}  \vdots & \hspace{2cm}  \vdots \\
	\big(-r_t,-r_t+1,\ldots,-(m_t+1)\big),& \big(-r_t,-r_t+1,\ldots,-(m_t+1)\big),& \big(-r_t,-r_t+1,\ldots,-(m_t+1)\big) &\\ 
	(r_t+1,r_t+2,\ldots,n) & (r_t+1,r_t+2,\ldots,n) & (r_t+1,r_t+2,\ldots,n)&
	\end{array}
\end{equation}
\end{small}
with
$$\begin{array}{llll}
 {\displaystyle \sum_{i=2}^t(r_i-m_i) \equiv 0 \; ({\rm mod} \ 2)}; & {\displaystyle \sum_{i=2}^t(r_i-m_i) \equiv 1 \; ({\rm mod} \ 2)}; &   {\displaystyle \sum_{i=1}^t(r_i-m_i) \equiv 0 \; ({\rm mod} \ 2).}& \end{array}$$
\end{itemize}
Some of these sequences may be empty, if $r_i=m_i$ or $r_i=m_{i+1}$ for some $i$, or if $m_i=0$.
\end{lem}

\begin{proof}
	The only difference with respect to the $B_n$ case is for the $0,1$ descents.  They depend on the relative positions of $\pm 1$ and $\pm 2$ in the window notation of $\g$. The following are the $D$-descent classes of all elements of $B_2$. We have that 
	$$\begin{array}{llll}
	D(\emptyset)&=\{[1,2], [-1,2]\}\\
	D(\{0\})&=\{[2,-1], [-2,-1]\}\\
	D(\{1\})&= \{[-2,1], [2,1]\}\\
	D(\{0,1\})&= \{[1,-2], [-1,-2]\}.
	\end{array}$$
	From this, the parity conditions $\sum_{i=1}^t(r_i-m_i) \equiv 0 $ or $\equiv 1$ (mod $2$), and Lemma \ref{Bcombclass} the result follows.
\end{proof}

\begin{rem}\label{rem33}
	Let us fix a subset of descents $M:=\{m_1,m_2,\ldots, m_t\}$. Consider the decompositions of $B_n$ and $D_n$ given by  Proposition \ref{2.1} by using left quotients.
%	$$ B_n= B_M B(M) \;\;\; {\rm and} \;\;\; D_n=D_M D(M).$$
	Recall that $|B_n|=2^n n!$ and that $|D_n|=2^{n-1} n!$. By looking at the Dynkin diagrams in Figure \ref{graphB} and Figure \ref{graphD}, it is easy to derive the following  equalities. 
	\begin{itemize}
	\item[{\small$\clubsuit$}] If $0 \in M$, then $|B(M)|=2\cdot |D(M)|$;
	\item[{\small$\clubsuit$}] If $0,1 \not\in M$, then $|B(M)|=|D(M)|$;
	\item[{\small$\clubsuit$}] If $0\not\in M$, and $1\in M$, then $|B(M)|=m_2\cdot |D(M)|$.
	\end{itemize}
\end{rem}

Now we make explicit these equalities by showing relations between $D$ and $B$ left quotients. 
%In order to do so, we consider the following linear extension to $B_n$ of the Bruhat order on $D_n$. More precisely, 
%for $\be, \g \in B_n$, we say that 
%$$\be \prec \g \Longleftrightarrow \ell_D(\g) < \ell_D(h),$$
%and we call it the {\em $D$-length order.} 

\begin{prop}\label{L1}
Let $0 \in M$. Then
\begin{itemize}
\item[i)] $B(M)$ splits into the disjoint union
$$B(M)=D(M) \uplus \bar{D}(M),$$ 
where $\bar{D}(M):=\{\bar{\g}=(-\g(1),\g(2),\ldots,\g(n)) \mid \g\in D(M)\}= \{\g\cdot s_0^B \mid \g\in D(M)\}$.
\item[ii)] Moreover
$$\sum_{\be \in B(M)} q^{\ell_D(\be)} = 2 \cdot { \sum_{\g \in D(M)} q^{\ell_D(\g)}}.$$
\end{itemize}
\end{prop}

\begin{proof}
Let $\g \in D(M)$. By Lemma \ref{Dcombclass} $\g$ is a shuffle of the sequences in (\ref{shuffleD-O}), and so it can also be obtained as a shuffle of the sequences in (\ref{shuffleB}). Hence $\g \in B(M)$. Now, let us change the sign to the first entry of $\g$, by getting $\bar{\g}$. We are changing the sign of $-r_i$, or of $r_i+1$ for $i \in [t]$, in one of the sequences in (\ref{shuffleD-O}). Note that this operation does not create a new $B$-descent for $\bar{\g}$. Hence $\bar{\g} \in B(M)\setminus D(M)$.  More precisely, $\bar{\g}$ can be obtained by shuffling the same sequences that give $\g$ where the twos involving $r_i$ are replaced either by $(-r_i+1,\ldots,-(m_i+1))$ and $(r_i, r_i+1,\ldots,m_i)$, or by $(-r_i-1,\ldots,-(m_i+1))$ and $(r_i+2,\ldots,m_i)$, depending if it is the sign of $-r_i$, or of $r_i+1$, that changes. All those sequences belong to (\ref{shuffleB}). So $i)$ follows by Remark \ref{rem33}.

Now, it is easy to see that for all  $\g \in D(M)$, one has $\ell_D(\g)=\ell_D(\bar{\g})$. To see that, suppose $\g(1)>0$. Then 
	\begin{eqnarray*}
	\inv(\bar{\g})=\inv(\g) - (\g(1)-1)\;\; {\rm and} \;\;\me_2(\bar{\g})=\me_2(\g)+(\g(1)-1), 
	\end{eqnarray*}
	and so the length $\ell_D$ is stable.  If $\g(1)<0$ a similar computations holds, hence	$ii)$ follows.
\end{proof}
Note that the two subsets $D(M)$ and $\bar{D}(M)$ are not isomorphic as posets, when they are considered as sub-posets of $(B(M),<_B)$, where $<_B$ denote the $B$-Bruhat order . An example is given  for $n=3$ and $M=\{0,2\}$.

\begin{prop}\label{L2}
Let $0,1 \not \in M$. Then 
\begin{itemize}
\item[i)] The map $\varphi : B(M) \longrightarrow D(M)$ defined by
$$ \beta \mapsto \left\{ \begin{array}{ll}
					\be, & \text{if} \;\; \be \in D_n;\\
					s^B_0 \cdot \be, & \text{otherwise},
				\end{array} \right.$$
is a bijection. 
\item[ii)] Moreover
$$\sum_{\be \in B(M)} q^{\ell_D(\be)} =  \sum_{\g \in D(M)} q^{\ell_D(\g)}.$$
\end{itemize}
\end{prop}
\begin{proof}
Let $\be \in B(M)$, it is a shuffle of the sequences in (\ref{shuffleB}). If $\be \in D_n$, then it is also a shuffle of the sequences in the first block of (\ref{shuffleD-01}). Hence $\be \in D(M)$. Now suppose that $\be \not\in D_n$. Since $0 \not\in \Des_B(\be^{-1})$, then $1 \in \be[n]$.  By multiplying on the left by $s_0^B$, we change the sign of $1$, and so the parity of $\beta$. Hence $s_0^B \cdot\beta \in D_n$. Actually, we obtain an element which is a shuffle of the sequences in second block of (\ref{shuffleD-01}). From Remark \ref{rem33} $i)$ follows. 

Since $\me_2(\g)=\me_2(\varphi(\g))$ and $\inv(\g)=\inv(\varphi(\g))$, one has $\ell_D(\g)=\ell_D(\varphi(\g))$, and so $ii)$ follows.
\end{proof}
The map $\varphi$ is not a poset isomorphism between $(B(M),<_B)$ and $(D(M),<_D)$, where $<_B$ and $<_D$ denote the corresponding Bruhat orders. When $n=3$, and $M=\{2\}$, $B(M)$ is a chain, while in $D(M)$ there are two elements not comparable. 

\begin{prop}\label{L3}
Let $0 \not\in M$, and $1\in M$. Then
\begin{itemize}
\item[i)] $B(M)$ splits as the disjoint union of the following $m_2$ subsets
$$B(M)=D_1(M)\uplus D_{12}(M) \uplus \ldots \uplus D_{12\ldots m_2}(M).$$
Each $D_{1\ldots i}(M)$ is in bijection with $D(M)$, and it is  recursively defined as follows:

1) $D_1(M)$ is obtained by shuffling the sequences defining $D(M)$ where $-1$ (if present) is replaced with $1$.

%2) $D_{12}(M)$ is obtained by shuffling the sequences defining $D_1(M)$ where: \\
%- $1$ and $\pm 2$ are switched if they are in the same sequence;\\
%- $\pm 2$ is replaced by $1$, and $1$ is replaced by $-2$, otherwise. This happens only when $1$ is at the beginning of a sequence.

2) For each $i\geq 2$, $D_{12\ldots i}(M)$ is obtained by shuffling the sequences defining $D_{12\ldots i-1}(M)$ where:
\begin{itemize}
\item[${\small \clubsuit}$] $1$ and $\pm i$ are swiched if they are in the same sequence; 
\item[${\small \clubsuit}$] $ i$ is replaced by $1$, and $1$ is replaced by $-i$, otherwise. This case happens when $1$ is at the beginning of a sequence of type $(1,-(i-1),\ldots, -2)$, and $i$ is the initial value of the sequence $(i,i+1,\ldots, m_2)$.
\end{itemize}
\item[ii)] Moreover
$$ \sum_{\be \in B(M)} q^{\ell_D(\be)} = [m_2]_q \cdot  \sum_{\g \in D(M)} q^{\ell_D(\g)}.$$
\end{itemize}
\end{prop}

\noindent Before writing down the proof let us consider an example.

\begin{example}
Consider $n=4$ and $M=\{1,3\}$. Then $D(M)$ is given by the shuffles of the following blocks of increasing sequences (written in column).
$$D(M)=\left\{
\begin{array}{rrrr}
(1) & (-1,2,3) & (-2,1) & (-3,-2,1) \\
(2,3); & (-4); & (3); & (4) \\
(4) &  & (-4) & \\
\end{array}\right\}
$$
Then $B(M)$ splits as disjoint union of the following three subsets:  
$$D_1(M)=\left\{
\begin{array}{rrrr}
(1) & (1,2,3) & (-2,1) & (-3,-2,1) \\
(2,3); & (-4); & (3); & (4) \\
(4) &  & (-4) & \\
\end{array}\right\}
$$
$$D_{12}(M)=\left\{
\begin{array}{rrrr}
(-2) & (2,1,3) & (1,-2) & (-3,1,-2) \\
(1,3); & (-4); & (3); & (4) \\
(4) &  & (-4) & \\
\end{array}\right\}
$$
$$D_{123}(M)=\left\{
\begin{array}{rrrr}
(-2) & (2,3,1) & (-3,-2) & (1,-3,-2) \\
(3,1); & (-4); & (1); & (4) \\
(4) &  & (-4) & \\
\end{array}\right\}.
$$
\end{example}
\medskip

\begin{proof}
The transformations defining $D_ {1\ldots i}(M)$ involve only the first two sequences of the three blocks of (\ref{shuffleD-1}). It is easy to see that $D_{1\ldots i}(M) \subseteq B(M)$ for all $i \in [m_2]$, and that $D_{1\ldots i}(M)$ and $D_{1\ldots j}(M)$ are disjoint if $i\neq j$. Hence the decomposition in $i)$ follows from Remark \ref{rem33}. 

Since changing $-1$ into $1$ in a signed permutation $\g$ affects neither $\inv(\g)$ nor $\me_2(\g)$, it follows that  
$$\sum_{\g \in D(M)} q^{\ell_D(\g)} = \sum_{\g \in D_{1}(M)} q^{\ell_D(\g)}.$$
Now let us show that for all $i\geq 2$ 
$$\sum_{\g \in D_{1\ldots i}(M)} q^{\ell_D(\g)} = q \sum_{\g \in D_{1\ldots i-1}(M)} q^{\ell_D(\g)}.$$ 
Let $\g \in D_{1\ldots i-1}(M)$. Consider the block in (\ref{shuffleD-1}) whose a particular shuffle gives $\g$. 

If $1$ and $\pm i$ are in the same sequence, it can be either of the form $(\ldots, 1,i,\ldots,m_2)$, or of the form $(-r_1,\ldots, -i,1\ldots,-2)$. Now consider the shuffle giving $\g$, where $1$ has been switched with $\pm i$. We get a new element $\bar{\g}\in D_{1\ldots i}(M)$. It is clear that  $\bar{\g}$ has one more inversion with respect to $\g$, and so the $D$-length go up by $1$. In fact, all other sequences in the block (whose shuffle gives $\g$) are made by elements that are either all bigger or all smaller of both $1$ and $\pm i$. Hence the difference between $\inv(\g)$ and $\inv(\bar{\g})$ depends only on the relative positions of $1$ and $\pm i$ within the same sequence. 

Suppose that $1$ and $i$ are not in the same sequence. This means that $1$ is at the beginning of the sequence $(1,-(i-1),\ldots,-2)$ and $i$ is at the beginning of the sequence $(i,i+1,\ldots,m_{2})$. So $\bar{\g} \in D_{1\ldots i}(M)$, the element corresponding to $\g$ after the switch, is obtained by shuffling a block that contains the following two sequences
$$(-i,-(i-1),\ldots,-2) \;\; \text{and}\;\; (1,i+1,\ldots,m_{2}).$$
Once again all other sequences of the block are made by elements that are either all smaller or bigger of both $1$ and $i$. 
The difference between the values of $\inv(\bar{\g})$ and $\inv(\g)$ depends only on the relative positions of $1$ and $i$. Hence
$\bar{\g}$ loses $i-2$ inversions with respect to $\g$ (the ones given by the $1$ at the beginning of the sequence), and $\me_2(\bar{\g})= \me_2(\g) +(i-1)$ thanks to $-i$. So $\ell_D(\bar{\g})=\ell_D(\g)+1$. 
\end{proof}

\section{Equidistribution over descent classes}

In this section we show generalizations of Theorem \ref{FS1} to Coxeter groups of type $B$ and $D$. We need the following classical result; see \cite[Theorem 3.1]{GG}, and \cite[Example 2.2.5]{StaEC1} for a proof.

\begin{thm}\label{stanley} 
Let $n \in \PP$ and $M=\{m_1,m_2,\ldots, m_t\}_{<}\subseteq [n-1]$. Then
\begin{eqnarray*}\sum_{\{\s \in S_n \mid \Des(\s^{-1})\subseteq M\}}q^{\maj(\s)} \ = \sum_{\{\s \in S_n \mid \Des(\s^{-1})\subseteq M\}}q^{\inv(\s)} = \begin{bmatrix} n \\ m_1, \; m_2-m_1, \ldots, n-m_t \end{bmatrix}_q.
\end{eqnarray*}
\end{thm}

\begin{thm}\label{mainB}
Let $n \in \PP$ and $M=\{m_1,m_2,\ldots,m_t\}_<\subseteq [0,n-1]$. Then
\begin{eqnarray}\label{th2}
\sum_{\{\be \in B_n \mid \Des_B(\be^{-1})\subseteq M\}} q^{\nmaj(\be)} &=& \sum_{\{\be \in B_n \mid \Des_B(\be^{-1})\subseteq M\}} q^{\ell_B(\be)} = \sum_{\{\be \in B_n \mid \Des_B(\be^{-1})\subseteq M\}}q^{\fmaj(\be)} \nonumber
\\
&=& \begin{bmatrix} n \\ m_1, \; m_2-m_1, \ldots, n-m_t \end{bmatrix}_q \cdot \displaystyle{\prod_{i=m_1+1}^{n} (1+q^i)}. 
\end{eqnarray}
\end{thm}

\begin{proof}
Let us denote by ${\rm Sh}(r_1,\ldots,r_t)$ the set of signed permutations obtained as shuffles of the sequences in (\ref{shuffleB}), with prescribed $r_1,\ldots,r_t$. From Theorem \ref{stanley} it follows that
\begin{equation}\label{crucialpoint}
\sum_{\beta \in {\rm Sh}(r_1,\ldots,r_t)}q^{\maj(\beta)}=\sum_{\beta \in {\rm Sh}(r_1,\ldots,r_t)}q^{\inv(\beta)}= \begin{bmatrix} n \\ m_1, \; r_1-m_1, \ldots, r_t-m_t,  \; n-r_t \end{bmatrix}_q. 
\end{equation}
In fact inversion number and major index of a shuffle depend only on the order of the elements in the shuffled sequences. From this, and the definitions of $\nmaj(\beta)=\maj(\beta)+\me_1(\beta)+\me_2(\beta)$ and of $\ell_B(\beta)=\inv(\beta)+\me_1(\beta)+\me_2(\beta)$, the first equality in (\ref{th2}) follows. The second equality and the sum have been computed in \cite{ABR2}. The symbol $\fmaj$ denote the {\em flag-major index} introduced by Adin and Roichman  in \cite{AR}.
\end{proof}
 
By the Principle of Inclusion-Exclusion we obtain
\begin{cor}\label{c1}
$$
\sum_{\{\beta \in B_n \mid \Des_B(\beta^{-1})=M\}}q^{\nmaj(\beta)}=\sum_{\{\beta \in B_n \mid \Des_B(\beta^{-1})=M\}}q^{\ell_B(\beta)}=\sum_{\{\beta \in B_n \mid \Des_B(\beta^{-1})=M\}}q^{\fmaj_B(\beta)}.$$
\end{cor}

The following lemma will be useful in the computation of our main result Theorem \ref{mainD}.

\begin{lem}\label{calcoloserie}
Let $n \in \PP$ and $M=\{m_1,m_2,\ldots,m_t\}_<\subseteq [0,n-1]$. Then
\begin{eqnarray*}
\sum_{\{\be \in B_n \mid \Des_B(\be^{-1})\subseteq M\}} q^{\ell_D(\be)} = \begin{bmatrix} n \\ m_1, \; m_2-m_1, \ldots, n-m_t \end{bmatrix}_q \cdot \displaystyle{\prod_{i=m_1}^{n-1} (1+q^i)}. 
\end{eqnarray*}
\end{lem}
\begin{proof}
Let $\be \in B(M)$. Recall that $\ell_D(\be)=\ell_B(\be)-\me_1(\be)$, and that $\ell_B(\beta)=\inv(\beta)+\sum_{\be(i)<0} |\be(i)|$.  Note that $\be(i)<0$ if and only if there exists a $j$ such that $m_j+1 
\leq |\be(i)| \leq r_j$. Therefore
\begin{eqnarray*}
\sum_{\be(i)<0} |\be(i)| & = & \sum_{i=1}^t (m_i+1) + \ldots + r_i\\
			 & = & \sum_{i=1}^t \left[(r_i - m_i)m_i + \frac{(r_i-m_i)(r_i-m_i+1)}{2}\right]\\
			 & = & \sum_{i=1}^t \frac{1}{2} (r_i-m_i)(r_i+m_i+1).
\end{eqnarray*}
Moreover $\me_1(\be) = \sum_{i=1}^t (r_i-m_i)$, and so 
\begin{eqnarray*}
\ell_D(\be)& = & \inv(\be)+ \sum_{i=1}^t \frac{1}{2}(r_i-m_i)(r_i+m_i+1) - (r_i-m_i)\\
		& = & \inv(\be) + \sum_{i=1}^t {r_i-m_i+1 \choose 2} + (r_i-m_i)(m_i-1)
\end{eqnarray*}
Hence by (\ref{crucialpoint}) 
\begin{eqnarray}
\sum_{\be \in B(M)} q^{\ell_{D}(\be)} & = & \sum_{r_1,\ldots,r_t} \ \sum_{\beta \in {\rm Sh}(r_1,\ldots, r_t)} q^{\inv(\be)} q^{\sum_{i=1}^t {r_i-m_i+1 \choose 2} + (r_i-m_i)} \nonumber \\
				     & = & \sum_{r_1,\ldots, r_t} \begin{bmatrix} n \\ m_1, \; r_1-m_1, \ldots, n-r_t \end{bmatrix}_q \cdot q^{\sum_{i=1}^t {r_i-m_i+1 \choose 2} + (r_i-m_i)(m_i-1) } \nonumber \\
				     & = & \begin{bmatrix} n \\ m_1, \; m_2-m_1, \ldots, n-m_t \end{bmatrix}_q \cdot \prod_{i=1}^t \sum_{r_i=m_i}^{m_{i+1}} \begin{bmatrix} m_{i+1}-m_i \\ r_i - m_i \end{bmatrix}_q \cdot q^{{r_i-m_i+1 \choose 2} + (r_i-m_i)(m_i-1)} \nonumber \\
				     				     & = & \begin{bmatrix} n \\ m_1, \; m_2-m_1, \ldots, n-m_t \end{bmatrix}_q \cdot \prod_{i=1}^t \prod_{j=m_i}^{m_{i+1}-1}(1+ q^j) \label{passo} \\
				     & = & \begin{bmatrix} n \\ m_1, \; m_2-m_1, \ldots, n-m_t \end{bmatrix}_q \cdot \prod_{j=m_1}^{n-1} (1+q^j) \nonumber
\end{eqnarray}
where the sum runs over $m_i\leq r_i \leq m_{i+1}$, and (\ref{passo}) is obtained by applying the $q$-binomial Theorem \ref{qbinomial} with  $x=q^{(m_i-1)}$.
\end{proof}

\begin{thm}\label{mainD}
Let $n \in \PP$ and $M=\{m_1,m_2,\ldots,m_t\}_<\subseteq [0,n-1]$. Then
\begin{eqnarray*}
\sum_{\g \in D(M)} q^{\dmaj(\g)}&=& \sum_{\g \in D(M)} q^{\ell_D(\g)}\\
&=& {\displaystyle \left\{\begin{array}{ll}
	\begin{bmatrix} n \\ m_1, \; m_2-m_1, \ldots, n-m_t \end{bmatrix}_q \cdot \displaystyle{\prod_{i=1}^{n-1} (1+q^i)} \; & {\rm if} \; 0 \in M;\\
        \begin{bmatrix} n \\ m_1, \; m_2-m_1, \ldots, n-m_t \end{bmatrix}_q \cdot \displaystyle{\prod_{i=m_1}^{n-1} (1+q^i)} \; & {\rm if} \; 0, 1 \not\in M;\\
	\begin{bmatrix} n \\ m_1, \; m_2-m_1, \ldots, n-m_t \end{bmatrix}_q \cdot \displaystyle{\frac{\prod_{i=1}^{n-1} (1+q^i)}{[m_2]_q}} \; & {\rm if} \; 0 \not\in M, \; {\rm and} \; 1 \in M.
	\end{array} 
	\right .}
\end{eqnarray*}
\end{thm}
\begin{proof} Once again the first equality follows from (\ref{crucialpoint}) and the definitions of $\dmaj$ and $\ell_D$.
The computation of the sum is now an easy application of Lemma \ref{calcoloserie}, together with Propositions \ref{L1}, \ref{L2}, and \ref{L3}.
\end{proof}
As corollary we obtain the desired generalization. 
\begin{cor}\label{c2}
$$
\sum_{\{\g \in D_n \mid \Des_D(\g^{-1})=M\}}q^{\dmaj(\g)}=\sum_{\{\g \in D_n \mid \Des_D(\g^{-1})=M\}}q^{\ell_D(\g)}.
$$
\end{cor}

\begin{rem}
If we replace  $\Des_B$ with  the usual descent set $\Des$,  Corollary \ref{c1} is still valid. It easily follows from Theorem \ref{mainB} since $\Des(\be^{-1})\subseteq M$ if and only of  $\Des_B(\be^{-1})\subseteq M \cup \{0\}$. Analogously, by replacing $\Des_D$ with $\Des$, Corollary \ref{c2} holds for the Coxeter group of type $D$.

\noindent The two corollaries are not true if as descent set one choose $\NDes$ for $B_n$ and $\DDes$ for $D_n$.
\end{rem}

\section{Symmetry of the joint distribution}

In this section we find generalizations of Foata-Sch\"utzenberger Theorem \ref{FS2}, Roselle Theorem \ref{Roselle}, and Gessel Theorem \ref{Gessel}.
\smallskip

The following  is an easy computation.

\begin{lem}\label{lemmino} Let $n\in \PP$. Then
\[\sum_{u \in B^J}p^{\me_1(u)} q^{\me_1(u)+\me_2(u)}=\sum_{S\subseteq [n]}p^{|S|} q^{\sum_{i\in S}i}=\prod_{i=1}^n (1+pq^i)=(-pq;q)_n.\]
Moreover 
\[\sum_{u \in D^J}p^{\me_1(u)+\epsilon(u)} q^{\me_2(u)}=\sum_{S\subseteq [n-1]} p^{|S|} q^{\sum_{i\in S}i}=\prod_{i=1}^{n-1} (1+pq^i)=(-pq;q)_{n-1}.\]
\end{lem}

\begin{prop}\label{sopra}
The distribution of $({\rm nmaj},\ell_B)$ over $B_n$ is symmetric, namely
\[B_n(t,q):=\sum_{\be \in B_n} t^{\nmaj(\be)}q^{\ell_B(\be)}=\sum_{\be \in B_n} t^{\ell_B(\be)}q^{\nmaj(\be)}\]
\end{prop}
\begin{proof} Let consider the decomposition (\ref{decomp}) of $B_n$. Let $u \in B^J$ (or $D^J$) and $\s \in S_n$. Then the following equalities hold
\[\maj(u\s)=\maj(\s) \;\; {\rm and} \;\; \inv(u\s)=\inv(u).\]
Moreover 
\[\me_1(u \s)=\me_1(u) \;\; {\rm and} \;\; \me_2(u\s)=\me_2(u).\] 
Then from Theorem \ref{FS2} it follows
\begin{eqnarray*}
\sum_{\be \in B_{n}} t^{\ell_B(\be)}q^{\nmaj(\g)} & = & \sum_{u
\in B^J} \sum_{\s \in S_{n}} t^{\inv(u \sigma) + \me_1(u \s)+ \me_{2}(u\sigma
)} q^{\maj(u \sigma )+\me_1(u \s) + \me_{2}(u\sigma)} \\
& = & \sum_{u \in B^J}t^{\me_1(u)+\me_2(u)}q^{\me_1(u)+\me_2(u)}\sum_{\s \in S_n}t^{\inv(\s)} q^{maj(\s)}\\
& = & \sum_{u \in B^J}t^{\me_1(u)+\me_2(u)}q^{\me_1(u)+\me_2(u)}\sum_{\s \in S_n}t^{\maj(\s)} q^{\inv(\s)}\\
 & = & \sum_{u
\in B^J} \sum_{\s \in S_{n}} t^{\maj(u \sigma) +\me_1(u\s)+\me_2(u\s)} q^{\inv(u \sigma)+ \me_1(u\s)+\me_2(u\s)}  \\
& = & \sum_{\be \in B_{n}} t^{\nmaj(\be)}q^{\ell_B(\be)}.
\end{eqnarray*}
\end{proof}

The analogous result holds for $D_n$. The proof is very similar to that of $B_n$ and is left to the reader.
\begin{prop}
The pair of statistics $(\dmaj,\ell_D)$ is symmetric, namely
\[D_n(t,q):=\sum_{\g \in D_n} t^{\dmaj(\g)}q^{\ell_D(\g)}=\sum_{\g \in D_n} t^{\ell_D(\g)}q^{\dmaj(\g)}.\]
\end{prop}

Note that, the flag-major index and the {\em $D$-major index} \cite{BC} do not share with $\nmaj$ and $\dmaj$ this symmetric distribution property. 
\smallskip

The following identities are generalizations of Theorem \ref{Roselle} of Roselle to $B_n$ and $D_n$. They easily follow from the proof of Proposition \ref{sopra}, Lemma \ref{lemmino}, and from Theorem \ref{Roselle}.
\begin{prop}[Roselle Identities for $B_n$ and $D_n$]
\begin{eqnarray*}
\sum_{n\geq 0} B_n(t,q) \frac{u^n}{(t;t)_n (q;q)_n  (-qt;qt)_{n}}&=&\frac{1}{(u;t,q)_{\infty,\infty}}, \ \  (B_0(t,q):=0);\\
1+ \sum_{n\geq 1} D_n(t,q) \frac{u^n}{(t;t)_n (q;q)_n (-qt;qt)_{n-1}}&=&\frac{1}{(u;t,q)_{\infty,\infty}}.
\end{eqnarray*}
\end{prop}
\smallskip

Similarly the following identities, which generalize Gessel formula, follow from the proof of Proposition \ref{sopra}, Lemma \ref{lemmino}, and Theorem \ref{Gessel}.  
\begin{prop}[Gessel Identities for $B_n$ and $D_n$]
\begin{eqnarray*}
\sum_{n\geq 0} \frac{u^n}{[n]_q !} \frac{\sum_{\be \in B_n} t^{\nmaj(\s)} q^{\ell_B(\be)} p^{\ndes(\be)}}{(-tqp;tq)_{n} (t;q)_{n+1}}&=&\sum_{k \geq 0} p^{k} e[u]_q e[tu]_q\cdots e[t^k u]_q;\\
\frac{1}{1-t} +\sum_{n\geq 1} \frac{u^n}{[n]_q !} \frac{\sum_{\g \in D_n} t^{\dmaj(\g)} q^{\ell_D(\g)} p^{\ddes(\g)}}{(-tqp;tq)_{n-1}(t;q)_{n+1}}&=&\sum_{k \geq 0} p^{k} e[u]_q e[tu]_q\cdots e[t^k u]_q.
\end{eqnarray*}
\end{prop}

\section{Concluding remarks}

As we mentioned along the paper, there exists another family of statistics, the {\em flag-statistics}, defined on Coxeter groups of type $B$, $D$ (see \cite{AR} and \cite{BC}), and more generally on complex reflection groups \cite{BaB}. Several generating functions involving flag-statistics have already been computed. In particular, we refer to the series of papers of Foata and Han \cite{FH2,FH3,FH5}, for a complete overview on the argument. 

We remark that among the series computed, none involve a combination of flag-statistics and length. This is why we conclude the paper with the following interesting proposal.

\begin{problem}
What kind of identities, generalizing the ones of Roselle and Gessel, might be obtained by using flag-statistics ?
\end{problem}

\end{document}